\documentclass[reqno]{amsart}

\usepackage{a4wide}
\usepackage{color}
\usepackage{mathrsfs}
\usepackage{mathtools}
\usepackage{tikz-cd}
\usepackage{amsmath}
\usepackage{amssymb}
\usepackage{bbm}
\numberwithin{equation}{section}
\usepackage[colorlinks,citecolor=green,linkcolor=red]{hyperref}
\usepackage{hyphenat}
\usepackage[latin1]{inputenc}

\newcommand{\B}{\mathbb{B}}
\newcommand{\V}{\mathbb{V}}
\newcommand{\N}{\mathbb{N}}
\newcommand{\R}{\mathbb{R}}
\newcommand{\sfd}{{\sf d}}
\renewcommand{\d}{{\mathrm d}}
\newcommand{\e}{{\rm e}}
\newcommand{\restr}[1]{\lower3pt\hbox{\(|_{#1}\)}}

\newcommand{\nchi}{{\raise.3ex\hbox{\(\chi\)}}}

\newcommand{\1}{\mathbbm 1}
\newcommand{\fr}{\penalty-20\null\hfill\(\blacksquare\)}
\newcommand{\X}{{\rm X}}
\newcommand{\Y}{{\rm Y}}

\newcommand{\ppi}{\boldsymbol{\pi}}

\newcommand{\LIP}{{\rm LIP}}
\newcommand{\Lip}{{\rm Lip}}
\newcommand{\lip}{{\rm lip}}

\newcommand{\Cyl}{{\rm Cyl}}
\newcommand{\Ch}{{\rm Ch}}

\newtheorem{theorem}{Theorem}[section]
\newtheorem{corollary}[theorem]{Corollary}
\newtheorem{lemma}[theorem]{Lemma}
\newtheorem{proposition}[theorem]{Proposition}
\newtheorem{definition}[theorem]{Definition}

\newtheorem{remark}[theorem]{Remark}

\title[Yet another proof of the density in energy of Lipschitz functions]{Yet another proof of the density \\ in energy of Lipschitz functions}

\author{Danka Lu\v{c}i\'{c}}
\address{Department of Mathematics and Statistics,
P.O.\ Box 35 (MaD), FI-40014 University of Jyvaskyla}
\email{danka.d.lucic@jyu.fi}

\author{Enrico Pasqualetto}
\address{Department of Mathematics and Statistics,
P.O.\ Box 35 (MaD), FI-40014 University of Jyvaskyla}
\email{enrico.e.pasqualetto@jyu.fi}

\begin{document}
\date{\today} 
\begin{abstract}
We provide a new, short proof of the density in energy of Lipschitz functions into the metric Sobolev space defined by using plans with barycenter
(and thus, a fortiori, into the Newtonian--Sobolev space). Our result covers first-order Sobolev spaces of exponent \(p\in(1,\infty)\), defined
over a complete separable metric space endowed with a boundedly-finite Borel measure.

Our proof is based on a completely smooth analysis: first we reduce the problem to the Banach space setting, where we consider smooth functions
instead of Lipschitz ones, then we rely on classical tools in convex analysis and on the superposition principle for normal \(1\)-currents.
Along the way, we obtain a new proof of the density in energy of smooth cylindrical functions in Sobolev spaces defined over a separable Banach
space endowed with a finite Borel measure.
\end{abstract}
\maketitle
\section{Introduction}
\subsection{General overview.}
Sobolev calculus on metric measure spaces has been a field of intense research activity for almost three decades.
In this paper, we focus on two approaches: the Sobolev space \(H^{1,p}(\X,\mu)\) obtained via relaxation and the Sobolev space \(W^{1,p}(\X,\mu)\) defined using plans with barycenter.
More specifically, fix a metric measure space \((\X,\sfd,\mu)\), i.e.\ \((\X,\sfd)\) is a complete separable metric space endowed with a boundedly-finite Borel measure \(\mu\geq 0\), and \(p\in(1,\infty)\). Then:
\begin{itemize}
\item By \(H^{1,p}(\X,\mu)\) we mean the Sobolev space via relaxation of Lipschitz functions, which was introduced by Ambrosio--Gigli--Savar\'{e} \cite{AmbrosioGigliSavare11}
as a variant of Cheeger's approach \cite{Cheeger00}; see Definition \ref{def:H1p}. We denote by \(|Df|_H\) the minimal relaxed slope of \(f\in H^{1,p}(\X,\mu)\).
\item By \(W^{1,p}(\X,\mu)\) we mean the Sobolev space defined using plans with barycenter, which was introduced in \cite{Sav19} after \cite{AmbrosioDiMarinoSavare,AmbrosioGigliSavare11,AmbrosioGigliSavare11-3};
see Definition \ref{def:W1p}. The notion of plan with barycenter we consider is essentially taken from \cite{Sav19}. We denote by \(|Df|_W\) the minimal weak upper gradient of \(f\in W^{1,p}(\X,\mu)\),
while \(\mathcal B_q(\X,\mu)\) is the space of plans \(\ppi\) with barycenter \({\sf Bar}(\ppi)\) in \(L^q(\mu)\), where \(q\) stands for the conjugate exponent of \(p\); see Definition \ref{def:plan}.
\end{itemize}
In this paper, we provide a new proof of the equivalence between \(H^{1,p}(\X,\mu)\) and \(W^{1,p}(\X,\mu)\), i.e.
\begin{equation}\label{eq:goal}
H^{1,p}(\X,\mu)=W^{1,p}(\X,\mu),\qquad|Df|_W=|Df|_H\;\text{ for every }f\in W^{1,p}(\X,\mu).
\end{equation}
It is worth underlining that we do not make any additional assumptions on \((\X,\sfd,\mu)\). In particular, we are not assuming that \(\mu\) is doubling nor the validity
of a Poincar\'{e} inequality; in the doubling-Poincar\'{e} framework, the equivalence was proved in \cite{Cheeger00,Shanmugalingam00}. We can rephrase \eqref{eq:goal} as follows:
\begin{center}
Lipschitz functions are \emph{dense in energy} in \(W^{1,p}(\X,\mu)\).
\end{center}
The result stated in \eqref{eq:goal} or variants of it were obtained earlier in the literature: 
\begin{itemize}
\item The first proof was obtained by Ambrosio--Gigli--Savar\'{e} in \cite{AmbrosioGigliSavare11-3},
where another class of plans (called \emph{test plans}) was used; see Remark \ref{rmk:cf_test_plan} for a comparison
with the notion of \(W^{1,p}(\X,\mu)\) we consider in this paper. The proof in \cite{AmbrosioGigliSavare11-3}
is based on the metric Hopf--Lax semigroup.
\item Savar\'{e} proved in \cite{Sav19} that \eqref{eq:goal} holds by using the von Neumann min-max principle and two representations of the dual Cheeger energy.
This approach is the closest to ours.
\item Eriksson-Bique proved in \cite{EB20} a variant of \eqref{eq:goal} via a more direct approximation technique. On the one hand, the result in \cite{EB20} is (a priori) weaker,
since it shows the identification between \(H^{1,p}(\X,\mu)\) and the Newtonian--Sobolev space \(N^{1,p}(\X,\mu)\); see the relative discussion in Section \ref{ss:consequences}.
On the other hand, \cite{EB20} covers also the case of the exponent \(p=1\).
\end{itemize}
Compared to the previous arguments, the novelty of our proof of \eqref{eq:goal} is that it relies on a purely smooth analysis. More precisely, since we can embed \((\X,\sfd)\)
isometrically into a Banach space, we can reduce the problem to the case where \(\X=\B\) itself is a Banach space. In this framework, we argue by using only smooth functions,
their Fr\'{e}chet differentials, classical tools in convex analysis, and normal \(1\)-currents. Neither Lipschitz functions nor other metric tools are actually needed.
\subsection{The proof strategy}
Up to a localisation argument and a Kuratowski embedding, we can reduce ourselves to addressing the problem in the case where \(\mu\) is finite and
\(\X=\B\) is a separable Banach space. We then consider the algebra \(\Cyl(\B)\) of \emph{cylindrical functions} (Definition \ref{def:cylindrical_fc}).
The advantage of working with cylindrical functions is that they are both smooth (of class \(C^\infty\)) and strongly dense in \(L^p(\mu)\). We define
the space \(H^{1,p}_{cyl}(\B,\mu)\) in analogy with \(H^{1,p}(\B,\mu)\), but using cylindrical functions in the relaxation procedure instead of Lipschitz functions.
We denote the corresponding minimal relaxed slope by \(|Df|_{H,cyl}\). Therefore, the new goal is to prove that
\begin{equation}\label{eq:goal_cyl}
H^{1,p}_{cyl}(\B,\mu)=W^{1,p}(\B,\mu),\qquad|Df|_W=|Df|_{H,cyl}\;\text{ for every }f\in W^{1,p}(\B,\mu).
\end{equation}
It is easy to show that \(H^{1,p}_{cyl}(\B,\mu)\subseteq H^{1,p}(\B,\mu)\subseteq W^{1,p}(\B,\mu)\) and \(|Df|_W\leq|Df|_H\leq|Df|_{H,cyl}\) for every \(f\in H^{1,p}_{cyl}(\B,\mu)\).
It follows that \eqref{eq:goal_cyl} implies \eqref{eq:goal}. To prove \eqref{eq:goal_cyl}, it suffices to check that
\begin{equation}\label{eq:goal_cyl_suff}
W^{1,p}(\B,\mu)\subseteq H^{1,p}_{cyl}(\B,\mu),\qquad\||Df|_{H,cyl}\|_{L^p(\mu)}\leq\||Df|_W\|_{L^p(\mu)}\;\text{ for all }f\in W^{1,p}(\B,\mu).
\end{equation}
In order to prove \eqref{eq:goal_cyl_suff}, we apply well-known results in convex analysis about Fenchel conjugates; see (the proof of) Theorem \ref{thm:equiv_Ban}.
Our arguments are strongly inspired by some ideas contained in Bouchitt\'{e}--Buttazzo--Seppecher's paper \cite{BBS97},
where Sobolev spaces on weighted Euclidean spaces were introduced. Roughly speaking, we consider the densely-defined
unbounded linear operator \(\d\colon L^p(\mu)\to L^p(\mu;\B^*)\) with domain \(D(\d)=\Cyl(\B)\), which assigns to
each function \(f\in\Cyl(\B)\) the \(\mu\)-a.e.\ equivalence class \(\d f\) of its Fr\'{e}chet differential. To prove the property \eqref{eq:goal_cyl_suff}
amounts to showing that \({\rm sc}^-\mathcal F(f)\leq\frac{1}{p}\||Df|_W\|_{L^p(\mu)}^p\)
for every \(f\in W^{1,p}(\B,\mu)\), where \({\rm sc}^-\mathcal F\) denotes the weak lower semicontinuous
envelope of the functional \(\mathcal F\colon L^p(\mu)\to[0,+\infty]\) given by
\[
\mathcal F(f)\coloneqq\frac{1}{p}\int\|\d_x f\|_{\B^*}^p\,\d\mu(x)\quad\text{ for every }f\in\Cyl(\B)
\]
and \(\mathcal F(f)\coloneqq+\infty\) otherwise (some extra care is needed when \(\mu\) is not fully supported). In order to achieve this goal, we need to prove the following statement:
given any \(L\in D(\d^*)\), there exists a plan \(\ppi\in\mathcal B_q(\B,\mu)\) such that \(\partial\ppi=(\d^*L)\mu\)
(see \eqref{eq:bdry_plan}) and \(\|{\sf Bar}(\ppi)\|_{L^q(\mu)}\leq\|L\|_{L^p(\mu;\B^*)^*}\). Here, we denote by \(\d^*\colon L^p(\mu;\B^*)^*\to L^q(\mu)\) the adjoint operator of \(\d\).
This is the content of Proposition \ref{prop:L_to_plan}, whose proof is based on Smirnov's superposition principle for normal \(1\)-currents \cite{Smirnov93}.
\subsection{Some additional comments}\label{ss:consequences}
With \eqref{eq:goal_cyl}, we recover a result by Savar\'{e} \cite{Sav19}, which states that cylindrical functions are dense in energy in \(W^{1,p}(\B,\mu)\); see also the paper \cite{FornasierSavareSodini22}.
Whereas Savar\'{e} obtains \eqref{eq:goal_cyl} as a consequence of the density in energy of Lipschitz functions, our proof goes in the opposite direction:
we prove directly \eqref{eq:goal_cyl}, then we obtain \eqref{eq:goal} as a corollary.
\medskip

Sobolev spaces over certain classes of \emph{weighted} Banach spaces (i.e.\ Banach spaces equipped with an arbitrary Borel measure) have been thoroughly
investigated in several articles. Weighted Euclidean spaces were studied e.g.\ in \cite{BBS97,Zhi00,Louet14,GP16-2,DMLP20,LPR21}, weighted Hilbert spaces (more generally,
weighted locally-\({\sf CAT}(\kappa)\) spaces) in \cite{DMGSP18}, and weighted reflexive Banach spaces in \cite{Sav19,FornasierSavareSodini22,Sodini22,PasRaj23}.
\medskip

We now consider the Newtonian--Sobolev space \(N^{1,p}(\X,\mu)\) introduced by Shanmugalingam \cite{Shanmugalingam00} (see also \cite{HKST15}) and the associated notion
of minimal weak upper gradient \(|Df|_N\). It is not too difficult to show that \(H^{1,p}(\X,\mu)\subseteq N^{1,p}(\X,\mu)\subseteq W^{1,p}(\X,\mu)\) and
\(|Df|_W\leq|Df|_N\leq|Df|_H\) for every \(f\in H^{1,p}(\X,\mu)\). Therefore, it follows directly from \eqref{eq:goal} that \(H^{1,p}(\X,\mu)=N^{1,p}(\X,\mu)\)
and that \(|Df|_N=|Df|_H\) for every \(f\in N^{1,p}(\X,\mu)\), in other words that Lipschitz functions are dense in energy in the Newtonian--Sobolev space; cf.\ with Remark \ref{rmk:cf_Newtonian}.
\medskip

We also mention that it seems that our proof strategy cannot be used to prove the identification between the spaces \(H^{1,1}\) and \(W^{1,1}\). Nevertheless, we do believe that it can
be adapted to show the equivalence of the different notions of functions of bounded variation, as well as to study various notions of Sobolev spaces of exponent \(p=\infty\).
These questions will be addressed in future works.
\subsection*{Acknowledgements}
The authors thank Luigi Ambrosio, Giuseppe Buttazzo, Toni Ikonen, and Giacomo Del Nin for the useful discussions
on the topics of this paper. The second named author has been supported by the MIUR-PRIN 202244A7YL
project ``Gradient Flows and Non-Smooth Geometric Structures with Applications to Optimization and Machine Learning''.
\section{Preliminaries}
Given \(p\in[1,\infty)\), we tacitly denote by \(q\coloneqq\frac{p}{p-1}\in(1,\infty]\) its conjugate exponent, and vice versa.
\subsection{Metric and measure spaces}
Given metric spaces \((\X,\sfd_\X)\), \((\Y,\sfd_\Y)\), we denote by \(C(\X;\Y)\) the space of continuous maps from \(\X\) to \(\Y\).
We endow its subset \(C_b(\X;\Y)\) consisting of bounded elements with the distance \(\sfd_{C_b(\X;\Y)}(\varphi,\psi)\coloneqq\sup_{x\in\X}\sfd_\Y(\varphi(x),\psi(x))\).
If \(\Y=\B\) is a Banach space, \(C_b(\X;\B)\) is a vector space and \(\sfd_{C_b(\X;\B)}\) is induced by the supremum norm \(\|\cdot\|_{C_b(\X;\B)}\).
We denote by \(\mathfrak M(\X)\) the set of (finite) signed Borel measures on \(\X\) and \(\mathfrak M_+(\X)\coloneqq\{\mu\in\mathfrak M(\X)\,:\,\mu\geq 0\}\).
For any \(\mu\in\mathfrak M(\X)\), we denote by \(\mu^+,\mu^-\in\mathfrak M_+(\X)\) the \emph{positive part} and the \emph{negative part} of \(\mu\), respectively. Recall that \(\mu=\mu^+-\mu^-\).
The total variation measure of \(\mu\in\mathfrak M(\X)\) is defined as \(|\mu|\coloneqq\mu^++\mu^-\in\mathfrak M_+(\X)\).
We endow \(\mathfrak M(\X)\) with the \emph{weak topology}, i.e.\ with the coarsest topology such that \(\mathfrak M(\X)\ni\mu\mapsto\int f\,\d\mu\)
is a continuous function for every \(f\in C_b(\X)\coloneqq C_b(\X;\R)\). We denote by \(\LIP(\X;\Y)\subseteq C(\X;\Y)\) the space of all Lipschitz maps from \(\X\) to \(\Y\).
Moreover, we define \(\LIP(\X)\coloneqq\LIP(\X;\R)\) and \(\LIP_b(\X)\coloneqq\LIP(\X)\cap C_b(\X)\). We call \(C_{bs}(\X)\) the space of \(f\in C(\X)\) whose support \({\rm spt}(f)\) is bounded
and we define \(\LIP_{bs}(\X)\coloneqq\LIP(\X)\cap C_{bs}(\X)\).
By \(\Lip(f)\) we mean the Lipschitz constant of \(f\in\LIP(\X)\), while its \emph{asymptotic slope} \(\lip_a(f)\colon\X\to[0,+\infty)\) is given by
\[
\lip_a(f)(x)\coloneqq\inf_{r>0}\Lip(f|_{B_r(x)})\quad\text{ for every }x\in\X.
\]
Let us now focus on the space \(C([0,1];\X)\) of curves. The \emph{evaluation maps} \(\e_\pm\colon C([0,1];\X)\to\X\) are the \(1\)-Lipschitz maps given by \(\e_+(\gamma)\coloneqq\gamma_1\)
and \(\e_-(\gamma)\coloneqq\gamma_0\). When \(\gamma\in\LIP([0,1];\X)\), the \emph{metric speed} \(|\dot\gamma_t|\coloneqq\lim_{h\to 0}\sfd_\X(\gamma_{t+h},\gamma_t)/|h|\) exists
for \(\mathcal L_1\)-a.e.\ \(t\in[0,1]\), where \(\mathcal L_1\) stands for the restriction of the one-dimensional Lebesgue measure to \([0,1]\). The \emph{length} of \(\gamma\) is defined
as \(\ell(\gamma)\coloneqq\int_0^1|\dot\gamma_t|\,\d t\). We say that \(\gamma\) is \emph{of constant speed} if \(|\dot\gamma|\) is \(\mathcal L_1\)-a.e.\ constant, so that
\(|\dot\gamma_t|=\ell(\gamma)\) for \(\mathcal L_1\)-a.e.\ \(t\in[0,1]\). By a \emph{plan} on \(\X\) we mean any measure \(\ppi\in\mathfrak M_+(C([0,1];\X))\). We define the
\emph{boundary} of \(\ppi\) as
\begin{equation}\label{eq:bdry_plan}
\partial\ppi\coloneqq(\e_+)_\#\ppi-(\e_-)_\#\ppi\in\mathfrak M(\X).
\end{equation}
Moreover, we define the Borel measure \(\|\ppi\|\geq 0\) on \(\X\) as \(\|\ppi\|\coloneqq\int\ell(\gamma)\,\gamma_\#\mathcal L_1\,\d\ppi(\gamma)\).
\medskip

One can also readily prove that, given any function \(f\in C_b(\X)\) such that \(f\geq 0\), it holds that
\begin{equation}\label{eq:lsc_int_plan}
\mathfrak M_+(C([0,1];\X))\ni\ppi\mapsto\int\!\!\!\int_0^1 f(\gamma_t)|\dot\gamma_t|\,\d t\,\d\ppi(\gamma)\quad\text{ is weakly lower semicontinuous.}
\end{equation}
\subsection{Banach spaces and \texorpdfstring{\(1\)}{1}-currents}
Let \(\B\) be a Banach space. For any function \(f\in C^\infty(\B)\), we denote by \(\d f\in C^\infty(\B;\B^*)\) its Fr\'{e}chet differential \(x\mapsto\d_x f\),
where \(\B^*\) is the dual of \(\B\). We define \(C^\infty_b(\B;\B^*)\coloneqq C^\infty(\B;\B^*)\cap C_b(\B;\B^*)\). If \(\V\) is a finite-dimensional Banach space,
then we also consider the space \(C_c^\infty(\V;\V^*)\) of all those \(\omega\in C^\infty_b(\V;\V^*)\) having compact support. The space of \emph{\(1\)-currents} in \(\V\)
is defined as the dual \({\bf M}_1(\V)\) of the normed space \((C_c^\infty(\V;\V^*),\|\cdot\|_{C_b(\V;\V^*)})\). When \(\V\) is a Euclidean space, these are the
\(1\)-currents in the sense of Federer--Fleming \cite{FedererFleming60}. The elements of \({\bf M}_1(\V)\) can be identified with the \(\V\)-valued Borel measures
on \(\V\), thus we can consider the total variation measure \(\|T\|\in\mathfrak M_+(\V)\) of every \(T\in{\bf M}_1(\V)\). Given any \(T\in{\bf M}_1(\V)\), we define
\[
\partial T(f)\coloneqq T(\d f)\quad\text{ for every }f\in C^\infty_c(\V).
\]
When the resulting operator \(\partial T\colon C^\infty_c(\V)\to\R\) -- which is called the \emph{boundary} of \(T\) -- belongs to the dual of \((C^\infty_c(\V),\|\cdot\|_{C_b(\V)})\),
we say that \(T\) is a \emph{normal \(1\)-current}. We denote by \({\bf N}_1(\V)\) the space of all normal \(1\)-currents in \(\V\). The boundary \(\partial T\) of each
\(T\in{\bf N}_1(\V)\) can be identified with a (finite) signed Borel measure on \(\V\). A \emph{subcurrent} of \(T\in{\bf M}_1(\V)\) is a current \(S\in{\bf M}_1(\V)\) such
that \(\|S\|+\|T-S\|=\|T\|\). By a \emph{cycle} of \(T\) we mean a subcurrent \(C\in{\bf N}_1(\V)\) of \(T\) such that \(\partial C=0\). We say that \(T\) is \emph{acyclic}
if its unique cycle is the null current. Then the following result holds (see e.g.\ \cite[Proposition 3.8]{PaolStep12}): for any \(T\in{\bf M}_1(\V)\), there exists
a cycle \(C\) of \(T\) such that \(T-C\) is acyclic. The following result states that acyclic normal \(1\)-currents are superpositions of curves:
\begin{theorem}[Superposition principle]\label{thm:Smirnov}
Let \(\V\) be a finite-dimensional Banach space. Then for every acyclic current \(T\in{\bf N}_1(\V)\)
there exists \(\ppi\in\mathfrak M_+(C([0,1];\V))\) concentrated on non-constant Lipschitz
curves of constant speed such that \((\e_+)_\#\ppi=(\partial T)^+\), \((\e_-)_\#\ppi=(\partial T)^-\), and \(\|T\|=\|\ppi\|\).
\end{theorem}
\begin{proof}
Since all norms on a finite-dimensional vector space are equivalent and the Euclidean norm is strictly convex, one can deduce the statement from Smirnov's results in \cite{Smirnov93}.
Alternatively, one can argue as follows: the metric \(1\)-currents on \(\V\) can be identified with the \(\mathbb V\)-valued Borel measures on \(\mathbb V\)
(see \cite[Lemma A.3]{PaoStep18}), thus the statement follows from \cite[Lemma 5.4]{PaolStep12}.
\end{proof}

We will focus on a distinguished class of smooth functions: the algebra of cylindrical functions.
\begin{definition}[Cylindrical function]\label{def:cylindrical_fc}
Let \(\B\) be a Banach space. Then we say that \(f\colon\B\to\R\) is a \emph{cylindrical function} if \(f=g\circ p\)
for some finite-dimensional Banach space \(\V\), some \(g\in C^\infty_c(\V)\), and some bounded linear map \(p\colon\B\to\V\).
We denote by \(\Cyl(\B)\) the space of cylindrical functions.
\end{definition}

It holds that \(f\in\LIP_b(\B)\) and \(\d f\in C^\infty_b(\B;\B^*)\) for every \(f\in\Cyl(\B)\).
Moreover, it holds that
\[
\lip_a(f)(x)=\|\d_x f\|_{\B^*}\quad\text{ for every }f\in\Cyl(\B)\text{ and }x\in\B.
\]
Given Banach spaces \(\B\), \(\V\) with \(\V\) finite-dimensional and a linear \(1\)-Lipschitz operator \(p\colon\B\to\V\),
we define the \emph{pullback} operator \(p^*\colon C^\infty_c(\V;\V^*)\to C^\infty_b(\B;\B^*)\) as follows: given any \(\omega\in C^\infty_c(\V;\V^*)\),
\begin{equation}\label{eq:def_pullback_1forms}
(p^*\omega)(x)\coloneqq p^{\rm adj}\circ((\omega\circ p)(x))\in\B^*\quad\text{ for every }x\in\B,
\end{equation}
where \(p^{\rm adj}\colon\V^*\to\B^*\) stands for the adjoint of \(p\), which is a linear \(1\)-Lipschitz operator. Hence,
\begin{equation}\label{eq:prop_pullback_1-forms}
\|(p^*\omega)(x)\|_{\B^*}\leq\|\omega(p(x))\|_{\V^*}\quad\text{ for every }x\in\B,
\end{equation}
thus in particular \(\|p^*\omega\|_{C_b(\B;\B^*)}\leq\|\omega\|_{C_b(\V;\V^*)}\). Notice that \(f\circ p\in\Cyl(\B)\) for every \(f\in C^\infty_c(\V)\),
and that \(p^*(\d f)=\d(f\circ p)\) thanks to the chain rule for Fr\'{e}chet differentials.
Given any \(\mu\in\mathfrak M_+(\B)\) and \(p\in[1,\infty)\), the \(\mu\)-a.e.\ equivalence
class \([p^*\omega]_\mu\) of \(p^*\omega\) belongs to the \emph{Lebesgue--Bochner space} \(L^p(\mu;\B^*)\),
which consists of all \(L^p(\mu)\)-integrable maps from \(\B\) to \(\B^*\) in the sense of Bochner \cite{DiestelUhl77}.
Notice that \((C^\infty_c(\V;\V^*),\|\cdot\|_{C_b(\V;\V^*)})\ni\omega\mapsto[p^*\omega]_\mu\in L^p(\mu;\B^*)\) is linear \(\mu(\B)\)-Lipschitz by \eqref{eq:prop_pullback_1-forms}.
\subsection{Metric Sobolev spaces}
By a \emph{metric measure space} \((\X,\sfd,\mu)\) we mean a complete and separable metric space \((\X,\sfd)\) together with a boundedly-finite Borel measure \(\mu\geq 0\) on \(\X\),
where ``boundedly-finite'' means that \(\mu(B)<+\infty\) whenever \(B\subseteq\X\) is a bounded Borel set.
Given any exponent $p\in[1,\infty]$, we denote by \((L^p(\mu),\|\cdot\|_{L^p(\mu)})\) the \(p\)-Lebesgue space on \((\X,\sfd,\mu)\). For any measurable function \(f\colon\X\to\R\),
we denote by \([f]_\mu\) its equivalence class up to \(\mu\)-a.e.\ equality. If \(\tilde\mu\) is a boundedly-finite Borel measure on \(\X\) such that \(\mu\leq\tilde\mu\), then we denote by
\({\rm ext}_{\tilde\mu}\colon L^p(\mu)\to L^p(\tilde\mu)\) the unique map satisfying \([{\rm ext}_{\tilde\mu}(f)]_\mu=f\) and \({\rm ext}_{\tilde\mu}(f)=0\) \(\tilde\mu\)-a.e.\ on
\(\big\{\frac{\d\mu}{\d\tilde\mu}=0\big\}\) for every \(f\in L^p(\mu)\).
\begin{remark}\label{rmk:proj_inject}{\rm
Let \((\X,\sfd,\mu)\) be a metric measure space with \({\rm spt}(\mu)=\X\) and let \(p\in[1,\infty]\). Then
\[
\big\{g\in C(\X):[g]_\mu\in L^p(\mu)\big\}\ni f\mapsto[f]_\mu\in L^p(\mu)\quad\text{ is injective.}
\]
Indeed, if two continuous functions agree \(\mu\)-a.e.\ on \(\X\), then they agree everywhere on \({\rm spt}(\mu)\).
\fr}\end{remark}
\subsubsection{Sobolev spaces via relaxation}
The first notion of metric Sobolev space we recall is based on a relaxation procedure.
The next definition, taken from \cite{AmbrosioGigliSavare11-3}, is a variant of Cheeger's one \cite{Cheeger00}.
\begin{definition}[Sobolev space via relaxation of Lipschitz functions]\label{def:H1p}
Let \((\X,\sfd,\mu)\) be a metric measure space and \(p\in(1,\infty)\). We define the \emph{Cheeger energy} functional \(\Ch\colon L^p(\mu)\to[0,+\infty]\) as
\[
\Ch(f)\coloneqq\inf\bigg\{\varliminf_n\frac{1}{p}\int\lip_a(f_n)^p\,\d\mu\;\bigg|\;(f_n)_n\subseteq\LIP_{bs}(\X),\,[f_n]_\mu\rightharpoonup f\text{ weakly in }L^p(\mu)\bigg\}.
\]
Then we define \(H^{1,p}(\X,\mu)\coloneqq\{f\in L^p(\mu):\Ch(f)<+\infty\}\). The \emph{minimal relaxed slope} of a given function \(f\in H^{1,p}(\X,\mu)\) is defined as the unique
\(|Df|_H\in L^p(\mu)^+\) satisfying \(\Ch(f)=\frac{1}{p}\int|Df|_H^p\,\d\mu\).
\end{definition}

On a weighted Banach space, one can give a similar definition using cylindrical functions instead:
\begin{definition}[Sobolev space via relaxation of cylindrical functions]\label{def:H1p_cyl}
Let \(\B\) be a separable Banach space and \(\mu\in\mathfrak M_+(\B)\). Let \(p\in(1,\infty)\) be a given exponent.
We define the \emph{cylindrical Cheeger energy} functional \(\Ch_{cyl}\colon L^p(\mu)\to[0,+\infty]\) as
\[
\Ch_{cyl}(f)\coloneqq\inf\bigg\{\varliminf_n\frac{1}{p}\int\|\d_x f_n\|_{\B^*}^p\,\d\mu\;\bigg|\;(f_n)_n\subseteq\Cyl(\B),\,[f_n]_\mu\rightharpoonup f\text{ weakly in }L^p(\mu)\bigg\}.
\]
Then we define \(H^{1,p}_{cyl}(\B,\mu)\coloneqq\{f\in L^p(\mu):\Ch_{cyl}(f)<+\infty\}\). The \emph{minimal cylindrical relaxed slope} of \(f\in H^{1,p}_{cyl}(\B,\mu)\)
is the unique \(|Df|_{H,cyl}\in L^p(\mu)^+\) satisfying \(\Ch_{cyl}(f)=\frac{1}{p}\int|Df|_{H,cyl}^p\,\d\mu\).
\end{definition}

Definition \ref{def:H1p_cyl} is a particular instance of the notion of metric Sobolev space via relaxation introduced in \cite{Sav19},
because \(\Cyl(\B)\) is a unital separating subalgebra of \(\LIP_b(\B)\) \cite[Example 2.1.19]{Sav19}. In particular, we know from \cite[Lemma 2.1.27]{Sav19}
that \([\Cyl(\B)]_\mu\) is dense in \(L^p(\mu)\). Moreover, the inclusion \(\Cyl(\B)\subseteq\LIP_b(\B)\), a standard cut-off argument, and the pointwise minimality properties
of minimal relaxed slopes (see \cite[Lemma 3.1.11]{Sav19}) ensure that \(H^{1,p}_{cyl}(\B,\mu)\subseteq H^{1,p}(\B,\mu)\) and
\begin{equation}\label{eq:Hcyl_in_H}
|Df|_H\leq|Df|_{H,cyl}\quad\text{ for every }f\in H^{1,p}_{cyl}(\B,\mu).
\end{equation}
\subsubsection{Sobolev spaces via plans}
The next notion was introduced in \cite[Definition 5.1.1]{Sav19} after \cite{AmbrosioDiMarinoSavare}.
\begin{definition}[Plan with barycenter]\label{def:plan}
Let \((\X,\sfd,\mu)\) be a metric measure space and \(q\in(1,\infty]\). We define \(\mathcal B_q(\X,\mu)\) as the set of all
\(\ppi\in\mathfrak M_+(C([0,1];\X))\) concentrated on \(\LIP([0,1];\X)\) such that:
\begin{itemize}
\item[\(\rm i)\)] \(\ppi\) has \emph{barycenter} in \(L^q(\mu)\), i.e.\ there exists a (unique) function \({\sf Bar}(\ppi)\in L^q(\mu)\) such that
\[
\int f\,{\sf Bar}(\ppi)\,\d\mu=\int\!\!\!\int_0^1 f(\gamma_t)|\dot\gamma_t|\,\d t\,\d\ppi(\gamma)\quad\text{ for every }f\in C_{bs}(\X).
\]
\item[\(\rm ii)\)] It holds that \((\e_\pm)_\#\ppi\ll\mu\) and \(\frac{\d(\e_\pm)_\#\ppi}{\d\mu}\in L^q(\mu)\).
\end{itemize}
\end{definition}

In Corollary \ref{cor:W1p_inv} and Proposition \ref{prop:L_to_plan}, we will need the following technical result about plans.
\begin{lemma}\label{lem:wlog_full_spt}
Let \((\X,\sfd,\mu)\) be a metric measure space such that \(S\coloneqq{\rm spt}(\mu)\neq\X\). Let \(C\subseteq\X\setminus S\) be a countable set. Let \(\tilde\mu\geq 0\) be a boundedly-finite Borel measure
on \(\X\) concentrated on \(S\cup C\) such that \(\tilde\mu|_S=\mu\). Fix any \(q\in(1,\infty]\) and \(\tilde\ppi\in\mathcal B_q(\X,\tilde\mu)\). Then \(\tilde\ppi\)-a.e.\ curve \(\gamma\) is either contained
in \(S\) or constant. In particular, \(\ppi\coloneqq\tilde\ppi|_{\LIP([0,1];S)}\in\mathcal B_q(\X,\mu)\), \({\rm ext}_{\tilde\mu}({\sf Bar}(\ppi))={\sf Bar}(\tilde\ppi)\), and \(\partial\ppi=\partial\tilde\ppi\).
\end{lemma}
\begin{proof}
Let \(\Gamma_{\rm const}\) be the set of constant curves in \(\X\) and define \(\Gamma_S\coloneqq\LIP([0,1];S)\). We claim that
\begin{equation}\label{eq:wlog_dull_spt}
\Gamma\coloneqq\big\{\gamma\in\LIP([0,1];\X)\;\big|\;\gamma_t\in S\cup C\text{ for }\mathcal L_1\text{-a.e.\ }t\in\{|\dot\gamma|>0\}\big\}\subseteq\Gamma_S\cup\Gamma_{\rm const}.
\end{equation}
Let us prove \eqref{eq:wlog_dull_spt}. Fix any \(\gamma\in\Gamma\setminus\Gamma_{\rm const}\). We aim to show that \(\gamma([0,1])\cap C=\varnothing\).
We argue by contradiction: suppose \(\gamma_a=x\) for some \(a\in[0,1]\) and \(x\in C\). Up to replacing \(\gamma\) with \(t\mapsto\gamma_{-t}\), we can assume \(a<1\), and we can find
\(b\in(a,1]\) such that \(r\coloneqq\sfd(\gamma_b,x)>0\) and \(\gamma([a,b])\subseteq\X\setminus S\). Since \(N\coloneqq\{\sfd(y,x):y\in C\}\) is countable and
\([a,b]\ni t\mapsto f(t)\coloneqq\sfd(\gamma_t,x)\) is Lipschitz, we deduce that \(f'=0\) holds \(\mathcal L_1\)-a.e.\ on \(f^{-1}(N)\). Moreover,
\(f^{-1}(\R\setminus N)\subseteq\X\setminus(S\cup C)\) and thus \(|\dot\gamma|=0\) holds \(\mathcal L_1\)-a.e.\ on \(f^{-1}(\R\setminus N)\). All in all, it follows that
\[
0<r=\sfd(\gamma_b,\gamma_a)=f(b)-f(a)=\int_a^b f'(t)\,\d t=\int_{f^{-1}(\R\setminus N)}f'(t)\,\d t\leq\int_{f^{-1}(\R\setminus N)}|\dot\gamma_t|\,\d t=0,
\]
which leads to a contradiction. Therefore, we proved that \(\gamma([0,1])\cap C=\varnothing\). Now consider the Lipschitz function \([0,1]\ni t\mapsto g(t)\coloneqq\sfd(\gamma_t,S)\).
Given that \(\gamma\in\Gamma\), we have that \(g=0\) holds \(\mathcal L_1\)-a.e.\ on \(\{|\dot\gamma|>0\}\), thus \(g'=0\) holds \(\mathcal L_1\)-a.e.\ on \(\{|\dot\gamma|>0\}\).
Since \(|g'|\leq|\dot\gamma|\) holds \(\mathcal L_1\)-a.e.\ on \([0,1]\), we conclude that \(g'=0\) holds \(\mathcal L_1\)-a.e.\ on \([0,1]\) and thus \(g\) is constant. Being \(\gamma\)
non-constant, we know that \(\mathcal L_1(\{|\dot\gamma|>0\})>0\), so that \(g=0\) on \([0,1]\). This means that \(\gamma\in\Gamma_S\), so that \eqref{eq:wlog_dull_spt} is proved.
Finally, observe that \(\int\!\!\int_0^1\1_{\X\setminus(S\cup C)}(\gamma_t)|\dot\gamma_t|\,\d t\,\d\tilde\ppi(\gamma)=\int_{\X\setminus(S\cup C)}{\sf Bar}(\tilde\ppi)\,\d\tilde\mu=0\),
whence it follows that for \(\tilde\ppi\)-a.e.\ \(\gamma\) it holds that \(\1_{\X\setminus(S\cup C)}(\gamma_t)|\dot\gamma_t|=0\) for \(\mathcal L_1\)-a.e.\ \(t\in[0,1]\). In particular,
we deduce that \(\tilde\ppi(\LIP([0,1];\X)\setminus\Gamma)=0\). Taking into account also \eqref{eq:wlog_dull_spt}, we have that the first part of the statement is proved. The last part
of the statement then easily follows.
\end{proof}

The following definition of Sobolev space via plans is taken from \cite[Definition 5.1.4]{Sav19}. Similar notions were previously introduced
in \cite{AmbrosioGigliSavare11,AmbrosioGigliSavare11-3,AmbrosioDiMarinoSavare}, see Remark \ref{rmk:cf_test_plan} for a quick comparison.
\begin{definition}[Sobolev space via plans]\label{def:W1p}
Let \((\X,\sfd,\mu)\) be a metric measure space and \(p\in(1,\infty)\). Then we declare that \(f\in L^p(\mu)\) belongs to \(W^{1,p}(\X,\mu)\) if there exists \(G\in L^p(\mu)^+\) such that
\[
\int f\,\d\partial\ppi\leq\int G\,{\sf Bar}(\ppi)\,\d\mu\quad\text{ for every }\ppi\in\mathcal B_q(\X,\mu).
\]
The \(\mu\)-a.e.\ minimal such \(G\) is called the \emph{minimal \(p\)-weak upper gradient} \(|Df|_{p,w}\in L^p(\mu)^+\) of \(f\).
\end{definition}

One can readily deduce from the definitions that \(H^{1,p}(\X,\mu)\subseteq W^{1,p}(\X,\mu)\) and that
\begin{equation}\label{eq:H_in_W}
|Df|_W\leq|Df|_H\quad\text{ for every }f\in H^{1,p}(\X,\mu).
\end{equation}
Indeed, for every \(f\in H^{1,p}(\X,\mu)\) there exists a sequence \((f_n)_n\subseteq\LIP_{bs}(\X)\) such that
\([f_n]_\mu\rightharpoonup f\) and \([\lip_a(f_n)]_\mu\rightharpoonup|Df|_H\) weakly in \(L^p(\mu)\). Therefore, for any given \(\ppi\in\mathcal B_q(\X,\mu)\) we can let \(n\to\infty\) in
\[
\int f_n\,\d\partial\ppi=\int\!\!\!\int_0^1\frac{\d}{\d t}f_n(\gamma_t)\,\d t\,\d\ppi(\gamma)\leq\int\!\!\!\int_0^1\lip_a(f_n)(\gamma_t)|\dot\gamma_t|\,\d t\,\d\ppi(\gamma)=
\int\lip_a(f_n)\,{\sf Bar}(\ppi)\,\d\mu,
\]
thus obtaining that \(\int f\,\d\partial\ppi\leq\int|Df|_H\,{\sf Bar}(\ppi)\,\d\mu\). This gives \(f\in W^{1,p}(\X,\mu)\) and \(|Df|_W\leq|Df|_H\).
\begin{remark}\label{rmk:compos_Sob}{\rm
Let \((\X,\sfd_\X,\mu_\X)\), \((\Y,\sfd_\Y,\mu_\Y)\) be metric measure spaces. Let \(S\coloneqq{\rm spt}(\mu_\X)\).
We call \(\phi\colon S\to\Y\) a \emph{short map} if it is \(1\)-Lipschitz and \(\phi_\#\mu_\X\leq\mu_\Y\). Define
\(\Phi\colon\LIP([0,1];S)\to\LIP([0,1];\Y)\) as \(\Phi(\gamma)_t\coloneqq\phi(\gamma_t)\)
for every \(\gamma\in\LIP([0,1];S)\) and \(t\in[0,1]\). Then the following claim can be readily checked: given any \(q\in(1,\infty)\) and \(\ppi\in\mathcal B_q(S,\mu_\X)\),
it holds that \(\Phi_\#\ppi\in\mathcal B_q(\Y,\mu_\Y)\) and
\begin{equation}\label{eq:compos_plan}
{\sf Bar}(\Phi_\#\ppi)=\frac{\d\phi_\#({\sf Bar}(\ppi)\mu_\X)}{\d\mu_\Y},\qquad\partial(\Phi_\#\ppi)=\phi_\#(\partial\ppi).
\end{equation}
Moreover, the map \(\phi\) induces via pre-composition a \(1\)-Lipschitz linear map \(\phi^*\colon L^p(\mu_\Y)\to L^p(\mu_\X)\). Using \eqref{eq:compos_plan}, one can easily
show that \(\phi^*W^{1,p}(\Y,\mu_\Y)\subseteq W^{1,p}(\X,\mu_\X)\) and \(|D(\phi^*f)|_W\leq\phi^*|Df|_W\) for every \(f\in W^{1,p}(\Y,\mu_\Y)\). It can also be readily
checked that \(\phi^*H^{1,p}(\Y,\mu_\Y)\subseteq H^{1,p}(\X,\mu_\X)\) and that \(|D(\phi^*f)|_H\leq\phi^*|Df|_H\) for every \(f\in H^{1,p}(\Y,\mu_\Y)\).
\fr}\end{remark}

The following technical statement is a direct consequence of Lemma \ref{lem:wlog_full_spt}.
\begin{corollary}\label{cor:W1p_inv}
Let \((\X,\sfd,\mu)\) be a metric measure space and let \(p\in(1,\infty)\). Let \(C\subseteq\X\setminus{\rm spt}(\mu)\)
be a countable set. Let \(\tilde\mu\geq 0\) be a boundedly-finite Borel measure on \(\X\) concentrated on \({\rm spt}(\mu)\cup C\)
such that \(\tilde\mu|_{{\rm spt}(\mu)}=\mu\). Then it holds that \(W^{1,p}(\X,\tilde\mu)=\big\{f\in L^p(\tilde\mu):[f]_\mu\in W^{1,p}(\X,\mu)\big\}\) and
\[
|Df|_W={\rm ext}_{\tilde\mu}(|D[f]_\mu|_W)\quad\text{ for every }f\in W^{1,p}(\X,\tilde\mu).
\]
\end{corollary}
\begin{proof}
On the one hand, taking \(\phi\coloneqq{\rm id}_\X\colon(\X,\tilde\mu)\to(\X,\mu)\) in Remark \ref{rmk:compos_Sob} we get \([f]_\mu\in W^{1,p}(\X,\mu)\)
and \(|D[f]_\mu|_W\leq[|Df|_W]_\mu\) for all \(f\in W^{1,p}(\X,\tilde\mu)\). Conversely, if \(f\in L^p(\tilde\mu)\) and \([f]_\mu\in W^{1,p}(\X,\mu)\),
then for every \(\tilde\ppi\in\mathcal B_q(\X,\tilde\mu)\) we deduce from Lemma \ref{lem:wlog_full_spt} that \(\ppi\coloneqq\tilde\ppi|_{\LIP([0,1];{\rm spt}(\mu))}\in\mathcal B_q(\X,\mu)\) and
\[
\int f\,\d\partial\tilde\ppi=\int[f]_\mu\,\d\partial\ppi\leq\int|D[f]_\mu|_W\,{\sf Bar}(\ppi)\,\d\mu=\int{\rm ext}_{\tilde\mu}\big(|D[f]_\mu|_W\big){\sf Bar}(\tilde\ppi)\,\d\tilde\mu,
\]
which implies that \(f\in W^{1,p}(\X,\tilde\mu)\) and \(|Df|_W\leq{\rm ext}_{\tilde\mu}(|D[f]_\mu|_W)\). The statement follows.
\end{proof}
\begin{remark}\label{rmk:wlog_fin_meas}{\rm
Let \((\X,\sfd,\mu)\) be a metric measure space and \(p\in(1,\infty)\). Let \((\Omega_n)_n\) be a sequence of open sets in \(\X\) such that \(\Omega_n\subseteq\Omega_{n+1}\)
for all \(n\in\N\) and \(\X=\bigcup_{n\in\N}\Omega_n\). Let \(f\in L^p(\mu)\) be such that \(f_n\in H^{1,p}(\X,\mu|_{\Omega_n})\) for every \(n\in\N\) and
\(s\coloneqq\sup_n\int|Df_n|_H^p\,\d\mu|_{\Omega_n}<+\infty\), where \(f_n\coloneqq[f]_{\mu|_{\Omega_n}}\). Then it holds that \(f\in H^{1,p}(\X,\mu)\) and \(\int|Df|_H^p\,\d\mu=s\).
This property can be proved by combining the locality of minimal relaxed slopes with a cut-off argument, see e.g.\ \cite[Proposition 2.17]{Cheeger00}.
\fr}\end{remark}
\section{Main results}
Let \(\B\) be a separable Banach space, \(\mu\in\mathfrak M_+(\B)\) a measure satisfying \({\rm spt}(\mu)=\B\), and \(p\in[1,\infty)\).
In view of Remark \ref{rmk:proj_inject}, we can identify \(\Cyl(\B)\) with a subspace of \(L^p(\mu)\), thus the Fr\'{e}chet differential induces an unbounded
linear operator \(\d\colon L^p(\mu)\to L^p(\mu;\B^*)\) with domain \(D(\d)=\Cyl(\B)\). Since \(\d\colon L^p(\mu)\to L^p(\mu;\B^*)\) is densely defined,
its adjoint operator \(\d^*\colon L^p(\mu;\B^*)^*\to L^q(\mu)\) is well-posed.
\begin{proposition}\label{prop:L_to_plan}
Let \(\B\) be a separable Banach space. Let \(\mu\in\mathfrak M_+(\B)\) be such that \({\rm spt}(\mu)=\B\). Let \(p\in[1,\infty)\) and \(L\in D(\d^*)\) be given.
Then there exists a plan \(\ppi\in\mathcal B_q(\B,\mu)\) such that
\[
\partial\ppi=(\d^*L)\mu,\qquad\|{\sf Bar}(\ppi)\|_{L^q(\mu)}\leq\|L\|_{L^p(\mu;\B^*)^*}.
\]
\end{proposition}
\begin{proof}
By \cite[Proposition 1.2.13]{Gigli14}, there exists a unique \(L^\infty(\mu)\)-linear map \(\ell\colon L^p(\mu;\B^*)\to L^1(\mu)\)
such that \(L(\omega)=\int\ell(\omega)\,\d\mu\) and \(|\ell(\omega)|\leq|L|\|\omega(\cdot)\|_{\B^*}\) in the \(\mu\)-a.e.\ sense for every \(\omega\in L^p(\mu;\B^*)\),
for some \(|L|\in L^q(\mu)^+\) satisfying \(\||L|\|_{L^q(\mu)}=\|L\|_{L^p(\mu;\B^*)^*}\). Since \(\B\) can be embedded linearly and isometrically into \(\ell^\infty\)
via a Kuratowski embedding and \(\ell^\infty\) has the metric approximation property (see e.g.\ \cite[Lemma 5.7]{PaolStep12}), we have that \(\B\) is the subspace of a
Banach space \(\tilde\B\) having the metric approximation property. Given that \(\mu\) is concentrated on a \(\sigma\)-compact set, we can find a sequence \((\tilde p_n)_n\)
of finite-rank \(1\)-Lipschitz linear operators \(\tilde p_n\colon\tilde\B\to\tilde\B\) such that \(\lim_n\|\tilde p_n(x)-x\|_{\tilde\B}=0\) holds for \(\mu\)-a.e.\ \(x\in\tilde\B\).
Now let us fix a separable closed subspace \(\hat\B\) of \(\tilde\B\) containing \(\B\cup\bigcup_{n\in\N}\tilde p_n(\B)\). For any \(n\in\N\), we denote by
\(\mathbb V_n\) the finite-dimensional Banach space \(\tilde p_n(\B)\subseteq\hat\B\) and we define the operator \(p_n\colon\B\to\V_n\) as \(p_n\coloneqq\tilde p_n|_\B\).
We define the \(1\)-current \(T_n\in{\bf M}_1(\V_n)\) as \(T_n(\omega)\coloneqq L([p_n^*\omega]_\mu)\) for every \(\omega\in C^\infty_c(\V_n;\V_n^*)\), where \(p_n^*\omega\)
is given by \eqref{eq:def_pullback_1forms}. We claim that \(T_n\in{\bf N}_1(\V_n)\) and
\begin{equation}\label{eq:L_to_plan_aux1}
\|T_n\|\leq(p_n)_\#(|L|\mu),\qquad\partial T_n=(p_n)_\#((\d^*L)\mu).
\end{equation}
To prove the first property in \eqref{eq:L_to_plan_aux1}, fix any open set \(\Omega\subseteq\V_n\) and an element \(\omega\in C^\infty_c(\V_n;\V_n^*)\) satisfying \({\rm spt}(\omega)\subseteq\Omega\)
and \(\|\omega(x)\|_{\V_n^*}\leq 1\) for every \(x\in\V_n\). Recalling \eqref{eq:prop_pullback_1-forms}, we can estimate
\[
|T_n(\omega)|\leq\int|\ell([p_n^*\omega]_\mu)|\,\d\mu\leq\int_{p_n^{-1}(\Omega)}|L|(x)\|\omega(p_n(x))\|_{\V_n^*}\,\d\mu(x)\leq(p_n)_\#(|L|\mu)(\Omega),
\]
whence it follows that \(\|T_n\|(\Omega)\leq(p_n)_\#(|L|\mu)(\Omega)\) and thus \(\|T_n\|\leq(p_n)_\#(|L|\mu)\). To prove the second property in \eqref{eq:L_to_plan_aux1},
notice that for every given function \(f\in C^\infty_c(\V_n)\) we can compute
\[
\partial T_n(f)=T_n(\d f)=L([p_n^*\d f]_\mu)=L(\d(f\circ p_n))=\int f\circ p_n\,\d^*L\,\d\mu=\int f\,\d(p_n)_\#((\d^*L)\mu),
\]
whence it follows that \(T_n\) is normal and \(\partial T_n=(p_n)_\#((\d^*L)\mu)\). All in all, the claim \eqref{eq:L_to_plan_aux1} is proved.
Since \((p_n)_\#(|L|\mu)\rightharpoonup|L|\mu\) and \((p_n)_\#(|\d^*L|\mu)\rightharpoonup|\d^*L|\mu\) weakly in \(\mathfrak M(\hat\B)\) by dominated
convergence theorem, Prokhorov's theorem gives that \((\|T_n\|)_n,(|\partial T_n|)_n\subseteq\mathfrak M_+(\hat\B)\)
are tight sequences.

For any \(n\in\N\), we choose a cycle \(C_n\) of \(T_n\) such that \(\tilde T_n\coloneqq T_n-C_n\in{\bf N}_1(\V_n)\)
is acyclic. Using Theorem \ref{thm:Smirnov}, we obtain a plan \(\ppi_n\in\mathfrak M_+(C([0,1];\hat\B))\),
concentrated on the set \(\Gamma\) of non-constant Lipschitz curves in \(\V_n\) of constant speed, such that
\(\|\tilde T_n\|=\|\ppi_n\|\) and \((\e_\pm)_\#\ppi_n=(\partial\tilde T_n)^\pm\). Now we follow the proof
of \cite[Lemma 4.11]{PaolStep12}. Since \(\|\ppi_n\|\leq\|T_n\|\) and \((\e_\pm)_\#\ppi_n\leq|\partial T_n|\)
for every \(n\in\N\), the sequences \((\|\ppi_n\|)_n,((\e_\pm)_\#\ppi_n)_n\subseteq\mathfrak M_+(\hat\B)\) are tight,
thus we can find compact subsets \((K_j)_j\) of \(\hat\B\) such that
\(\|\ppi_n\|(\hat\B\setminus K_j)\leq 4^{-j}\) and \(((\e_+)_\#\ppi_n)(\hat\B\setminus K_j)\leq 2^{-j}\) for all \(j,n\in\N\). We also define
\[
\Gamma_j\coloneqq\{\gamma\in\Gamma\;|\;\ell(\gamma)\leq 2^j\}\cap\e_+^{-1}(K_j)\cap\bigcap_{k>j}\tilde\Gamma_k,
\]
where \(\tilde\Gamma_k\coloneqq\{\gamma\in\Gamma\,:\,\mathcal L_1(\gamma^{-1}(\hat\B\setminus K_k))\leq 2^{-k}/\ell(\gamma)\}\).
Since \(2^{-k}\ppi_n(\Gamma\setminus\tilde\Gamma_k)\leq\|\ppi_n\|(\hat\B\setminus K_k)\leq 4^{-k}\) and
\(2^j\ppi_n(\{\gamma\in\Gamma\,:\,\ell(\gamma)>2^j\})\leq\|\ppi_n\|(\hat\B)\leq\int|L|\,\d\mu\eqqcolon m\) for every \(n\in\N\), we deduce that
\begin{equation}\label{eq:L_to_plan_aux2}
\ppi_n(\Gamma\setminus\Gamma_j)\leq\frac{m}{2^j}+((\e_+)_\#\ppi_n)(\hat\B\setminus K_j)+\sum_{k>j}\ppi_n(\Gamma\setminus\tilde\Gamma_k)
\leq\frac{m}{2^j}+\frac{1}{2^j}+\sum_{k>j}\frac{1}{2^k}=\frac{m+2}{2^j}
\end{equation}
for all \(j,n\in\N\). We now show that \(\Gamma_j\) is a precompact subset of \(C([0,1];\hat\B)\). Fix \((\gamma^i)_i\subseteq\Gamma_j\). Then:
\begin{itemize}
\item Suppose \(\lim_i\ell(\gamma^i)=0\). Since \(((\gamma^i)_1)_i\subseteq K_j\) and \(K_j\) is compact, \(x\coloneqq\lim_i(\gamma^i)_1\in K_j\) exists, up to subsequence.
Hence, \((\gamma^i)_i\) converges uniformly to the curve constantly in \(x\).
\item Suppose \(\varlimsup_i\ell(\gamma^i)>0\), so that \(c\coloneqq\inf_i\ell(\gamma^i)>0\) up to subsequence. Since \(\Lip(\gamma^i)\leq 2^j\) and
\(\mathcal L_1((\gamma^i)^{-1}(\hat\B\setminus K_k))\leq 2^{-k}/c\) for every \(i\in\N\) and \(k>j\), the sequence \((\gamma^i)_i\) has a uniformly converging
subsequence by the Arzel\`{a}--Ascoli theorem \cite[Proposition 2.1]{PaolStep12}.
\end{itemize}
All in all, we proved that each \(\Gamma_j\) is precompact, thus \eqref{eq:L_to_plan_aux2} implies that \((\ppi_n)_n\subseteq\mathfrak M_+(C([0,1];\hat\B))\)
is a tight sequence. Since \(\ppi_n(C([0,1];\hat\B))\leq|\partial T_n|(\hat\B)\leq\int|\d^*L|\,\d\mu\) for every \(n\in\N\), we know from Prokhorov's theorem that
\(\ppi_n\rightharpoonup\hat\ppi\) for some \(\hat\ppi\in\mathfrak M_+(C([0,1];\hat\B))\), up to subsequence. Note that
\[\begin{split}
\bigg|\int\!\!\!\int_0^1 f(\gamma_t)|\dot\gamma_t|\,\d t\,\d\hat\ppi(\gamma)\bigg|&\leq\varliminf_n\int\!\!\!\int_0^1|f|(\gamma_t)\|\dot\gamma_t\|_{\V_n}\,\d t\,\d\ppi_n(\gamma)=\varliminf_n\int|f|\,\d\|\ppi_n\|\\
&\leq\lim_n\int|f|\circ p_n\,|L|\,\d\mu=\int|f||L|\,\d\mu\leq\|f\|_{L^p(\mu)}\|L\|_{L^p(\mu;\B^*)^*}
\end{split}\]
for every \(f\in C_{bs}(\hat\B)\) thanks to \eqref{eq:lsc_int_plan}, \eqref{eq:L_to_plan_aux1}, and the dominated convergence theorem. This implies that \(\hat\ppi\) has barycenter in \(L^q(\mu)\) and
\(\|{\sf Bar}(\hat\ppi)\|_{L^q(\mu)}\leq\|L\|_{L^p(\mu;\B^*)^*}\). Given that \((\e_\pm)_\#\ppi_n\rightharpoonup(\e_\pm)_\#\hat\ppi\) and \((\e_\pm)_\#\ppi_n\leq(p_n)_\#(|\d^*L|\mu)\rightharpoonup|\d^*L|\mu\),
we have that \((\e_\pm)_\#\hat\ppi\ll\mu\) and \(\frac{\d(\e_\pm)_\#\hat\ppi}{\d\mu}\leq|\d^*L|\in L^q(\mu)\). Therefore, it holds \(\hat\ppi\in\mathcal B_q(\hat\B,\mu)\). Also,
\(\partial\ppi_n=(\e_+)_\#\ppi_n-(\e_-)_\#\ppi_n\rightharpoonup(\e_+)_\#\hat\ppi-(\e_-)_\#\hat\ppi=\partial\hat\ppi\) and
\(\partial\ppi_n=\partial\tilde T_n=\partial T_n=(p_n)_\#((\d^*L)\mu)\rightharpoonup(\d^*L)\mu\), so that \(\partial\hat\ppi=(\d^*L)\mu\). Thanks to the last part of Lemma \ref{lem:wlog_full_spt},
we then conclude that \(\ppi\coloneqq\hat\ppi|_{\LIP([0,1];\B)}\in\mathcal B_q(\B,\mu)\) verifies the statement.
\end{proof}
\begin{remark}\label{rem:about_currents}{\rm
Alternatively, in the proof of Proposition \ref{prop:L_to_plan} we could have used metric \(1\)-currents in the sense of Ambrosio--Kirchheim \cite{AmbrosioKirchheim00}
and Paolini--Stepanov's metric version of the superposition principle \cite{PaolStep12,PaolStep13}.
We opted for the proof we presented for two reasons: first, the current \(T\) we want to associate to \(L\) is defined on cylindrical functions, but it is not obvious
how to extend it to Lipschitz functions; second, extending \(T\) to Lipschitz functions is de facto unnecessary.
\fr}\end{remark}

With Proposition \ref{prop:L_to_plan} at disposal, we prove the equivalence result for weighted Banach spaces:
\begin{theorem}[Equivalence of Sobolev spaces on weighted Banach spaces]\label{thm:equiv_Ban}
Let \(\B\) be a separable Banach space, \(\mu\in\mathfrak M_+(\B)\), and \(p\in(1,\infty)\). Then it holds that \(H^{1,p}_{cyl}(\B,\mu)=W^{1,p}(\B,\mu)\) and
\[
|Df|_W=|Df|_{H,cyl}\quad\text{ for every }f\in W^{1,p}(\B,\mu).
\]
\end{theorem}
\begin{proof}
If \({\rm spt}(\mu)=\B\), define \(\tilde\mu\coloneqq\mu\). Otherwise, fix any dense sequence \((x_n)_n\) in \(\B\setminus{\rm spt}(\mu)\) and call
\(\tilde\mu\coloneqq\mu+\sum_n 2^{-n}\delta_{x_n}\in\mathfrak M_+(\B)\). By Remark \ref{rmk:proj_inject}, it makes sense to define \(\mathcal F\colon L^p(\tilde\mu)\to[0,+\infty]\) as
\[
\mathcal F(f)\coloneqq\frac{1}{p}\int\|\d_x f\|_{\B^*}^p\,\d\tilde\mu(x)\quad\text{ if }f\in\Cyl(\B)\subseteq L^p(\tilde\mu)
\]
and \(\mathcal F(f)\coloneqq+\infty\) otherwise. Consider its Fenchel conjugate and its double Fenchel conjugate, i.e.
\[
\mathcal F^*(g)\coloneqq\sup_{\tilde h\in L^p(\tilde\mu)}\int g\tilde h\,\d\tilde\mu-\mathcal F(\tilde h),\qquad
\mathcal F^{**}(f)\coloneqq\sup_{\tilde g\in L^q(\tilde\mu)}\int f\tilde g\,\d\tilde\mu-\mathcal F^*(\tilde g)
\]
for every \(g\in L^q(\tilde\mu)\) and \(f\in L^p(\tilde\mu)\). Since \(\mathcal F\) is convex, we know e.g.\ from \cite[Theorem 5]{Rockafellar74} that \(\mathcal F^{**}\) coincides
with the weak lower semicontinuous envelope \({\rm sc}^-\mathcal F\colon L^p(\tilde\mu)\to[0,+\infty]\) of \(\mathcal F\). Moreover, letting \(\phi\colon L^p(\tilde\mu;\B^*)\to[0,+\infty)\)
be given by \(\phi\coloneqq\frac{1}{p}\|\cdot\|_{L^p(\tilde\mu;\B^*)}^p\), we have that \(\mathcal F=\phi\circ\d\) and thus
\[
\mathcal F^*(g)=\inf\bigg\{\frac{1}{q}\|L\|_{L^p(\tilde\mu;\B^*)^*}^q\;\bigg|\;L\in D(\d^*),\,\d^*L=g\bigg\}\quad\text{ for every }g\in L^q(\tilde\mu)
\]
thanks to \cite[Theorem 5.1]{BBS97}. Indeed, \(\phi\) is convex and continuous, and its Fenchel conjugate is given by \(\phi^*=\frac{1}{q}\|\cdot\|_{L^p(\tilde\mu;\B^*)^*}^q\).
Combining this with Proposition \ref{prop:L_to_plan} and Young's inequality, we obtain
\[\begin{split}
\Ch_{cyl}(\tilde f)&={\rm sc}^-\mathcal F(\tilde f)=\sup_{g\in L^q(\tilde\mu)}\bigg(\int\tilde f g\,\d\tilde\mu-\inf_{\substack{L\in D(\d^*): \\ \d^*L=g}}\frac{1}{q}\|L\|_{L^p(\tilde\mu;\B^*)^*}^q\bigg)\\
&\leq\sup_{g\in L^q(\tilde\mu)}\bigg(\int\tilde f g\,\d\tilde\mu-\inf_{\substack{\ppi\in\mathcal B_q(\B,\tilde\mu): \\ \partial\ppi=g\tilde\mu}}\frac{1}{q}\|{\sf Bar}(\ppi)\|_{L^q(\tilde\mu)}^q\bigg)\\
&=\sup_{\ppi\in\mathcal B_q(\B,\tilde\mu)}\bigg(\int\tilde f\,\d\partial\ppi-\frac{1}{q}\|{\sf Bar}(\ppi)\|_{L^q(\tilde\mu)}^q\bigg)\\
&\leq\sup_{\ppi\in\mathcal B_q(\B,\tilde\mu)}\bigg(\int|D\tilde f|_W\,{\sf Bar}(\ppi)\,\d\tilde\mu-\frac{1}{q}\|{\sf Bar}(\ppi)\|_{L^q(\tilde\mu)}^q\bigg)\\
&\leq\frac{1}{p}\int|D\tilde f|_W^p\,\d\tilde\mu\quad\text{ for every }\tilde f\in W^{1,p}(\B,\tilde\mu).
\end{split}\]
This gives that \(W^{1,p}(\B,\tilde\mu)\subseteq H^{1,p}_{cyl}(\B,\tilde\mu)\) and \(\int|D\tilde f|_{H,cyl}^p\,\d\tilde\mu\leq\int|D\tilde f|_W^p\,\d\tilde\mu\)
for all \(\tilde f\in W^{1,p}(\B,\tilde\mu)\). 

Now fix any function \(f\in W^{1,p}(\B,\mu)\) and define \(\tilde f\coloneqq{\rm ext}_{\tilde\mu}(f)\in L^p(\tilde\mu)\). The first part of the proof, Corollary \ref{cor:W1p_inv}, and Remark \ref{rmk:compos_Sob}
ensure that \(\tilde f\in W^{1,p}(\B,\tilde\mu)\), as well as \(f\in H^{1,p}_{cyl}(\B,\mu)\) and
\[
\int|Df|_{H,cyl}^p\,\d\mu\leq\int|D\tilde f|^p_{H,cyl}\,\d\tilde\mu\leq\int|D\tilde f|^p_W\,\d\tilde\mu=\int|Df|^p_W\,\d\mu.
\]
Since \(|Df|_W\leq|Df|_{H,cyl}\) by \eqref{eq:Hcyl_in_H}, \eqref{eq:H_in_W} and \(H^{1,p}_{cyl}(\B,\mu)\subseteq W^{1,p}(\B,\mu)\), the statement follows.
\end{proof}

The equivalence result for arbitrary metric measure spaces easily follows:
\begin{theorem}[Equivalence of metric Sobolev spaces]\label{thm:equiv_metric_Sob}
Let \((\X,\sfd,\mu)\) be a metric measure space and \(p\in(1,\infty)\). Then it holds that \(H^{1,p}(\X,\mu)=W^{1,p}(\X,\mu)\) and
\[
|Df|_W=|Df|_H\quad\text{ for every }f\in W^{1,p}(\X,\mu).
\]
\end{theorem}
\begin{proof}
In view of \eqref{eq:H_in_W}, it suffices to show that \(f\in H^{1,p}(\X,\mu)\) and \(\int|Df|^p_H\,\d\mu\leq\int|Df|^p_W\,\d\mu\) for every fixed \(f\in W^{1,p}(\X,\mu)\).
Fix a linear isometric embedding \(\iota\colon\X\hookrightarrow\B\) into some separable Banach space \(\B\)
and call \(\tilde\mu\coloneqq\iota_\#\mu\). Define \(\tilde\Omega_n\coloneqq B_n(0)\subseteq\B\) and
\(\Omega_n\coloneqq\iota^{-1}(\tilde\Omega_n)\) for every \(n\in\N\). Both
\(\iota_n\coloneqq\iota|_{\Omega_n}\colon({\rm spt}(\mu_n),\mu_n)\to(\B,\tilde\mu_n)\) and
\(\phi_n\coloneqq(\iota|_{{\rm spt}(\mu_n)})^{-1}\colon({\rm spt}(\tilde\mu_n),\tilde\mu_n)\to(\X,\mu_n)\)
are short maps (in the sense of Remark \ref{rmk:compos_Sob}), where we set \(\mu_n\coloneqq\mu|_{\Omega_n}\)
and \(\tilde\mu_n\coloneqq\tilde\mu|_{\tilde\Omega_n}=\iota_\#\mu_n\). Hence:
\begin{itemize}
\item \(f_n\coloneqq[f]_{\mu_n}\in W^{1,p}(\X,\mu_n)\) and \(\int|Df_n|^p_W\,\d\mu_n\leq\int|Df|^p_W\,\d\mu\)
thanks to Remark \ref{rmk:compos_Sob}.
\item \(\tilde f_n\coloneqq\phi_n^*f_n\in W^{1,p}(\B,\tilde\mu_n)\) and
\(\int|D\tilde f_n|^p_W\,\d\tilde\mu_n\leq\int|Df_n|^p_W\,\d\mu_n\) again by Remark \ref{rmk:compos_Sob}.
\item \(\tilde f_n\in H^{1,p}(\B,\tilde\mu_n)\) and
\(\int|D\tilde f_n|^p_H\,\d\tilde\mu_n=\int|D\tilde f_n|^p_W\,\d\tilde\mu_n\) by Theorem \ref{thm:equiv_Ban},
\eqref{eq:Hcyl_in_H}, and \eqref{eq:H_in_W}.
\item \(f_n=\iota_n^*\tilde f_n\in H^{1,p}(\X,\mu_n)\) and
\(\int|Df_n|^p_H\,\d\mu_n\leq\int|D\tilde f_n|^p_H\,\d\tilde\mu_n\) thanks to Remark \ref{rmk:compos_Sob}.
\end{itemize}
All in all, we proved \(f_n\in H^{1,p}(\X,\mu_n)\) and \(\int|Df_n|^p_H\,\d\mu_n\leq\int|Df|^p_W\,\d\mu\)
for all \(n\in\N\). By Remark \ref{rmk:wlog_fin_meas}, we conclude that
\(f\in H^{1,p}(\X,\mu)\) and \(\int|Df|^p_H\,\d\mu=\sup_n\int|Df_n|^p_H\,\d\mu_n\leq\int|Df|^p_W\,\d\mu\).
\end{proof}

To conclude, we briefly comment on the equivalence with other notions of metric Sobolev space:
\begin{remark}[Equivalence with Newtonian--Sobolev spaces]\label{rmk:cf_Newtonian}{\rm
Let \((\X,\sfd,\mu)\) be a metric measure space and \(p\in(1,\infty)\). We denote by \(N^{1,p}(\X,\mu)\) the \emph{Newtonian--Sobolev space}, in the sense of \cite{Shanmugalingam00}.
It is known that \(H^{1,p}(\X,\mu)\subseteq N^{1,p}(\X,\mu)\subseteq W^{1,p}(\X,\mu)\) and that \(|Df|_W\leq|Df|_N\leq|Df|_H\) for every \(f\in H^{1,p}(\X,\mu)\),
where \(|Df|_N\) denotes the minimal \(p\)-weak upper gradient in the sense of Newtonian--Sobolev spaces; see \cite{HKST15} for the first inclusion and \cite{AmbrosioDiMarinoSavare,Sav19}
for the second. Consequently, Theorem \ref{thm:equiv_metric_Sob} implies \(N^{1,p}(\X,\mu)=H^{1,p}(\X,\mu)\) and \(|Df|_N=|Df|_H\) for every \(f\in N^{1,p}(\X,\mu)\).
\fr}\end{remark}

\begin{remark}[Equivalence with Sobolev spaces via test plans]\label{rmk:cf_test_plan}{\rm
In the paper \cite{AmbrosioGigliSavare11-3}, \(W^{1,p}(\X,\mu)\) was defined in a different way, in terms of \emph{test plans}. Nevertheless, it follows from the
proof arguments of \cite[Theorems 8.5 and 9.4]{AmbrosioDiMarinoSavare} that the notion in \cite{AmbrosioGigliSavare11-3} coincides with ours. We omit the details.
\fr}\end{remark}
\end{document}